\documentclass[a4paper,12pt]{amsart}
\usepackage[english]{babel}
\usepackage{amsmath,amsthm,amssymb,amsfonts}
\usepackage{multicol}
\usepackage[titletoc]{appendix}
\usepackage{newpxmath}
\usepackage[pdftex]{color}
\usepackage[bookmarks=true,hyperindex,pdftex,colorlinks, citecolor=blue,linkcolor=blue, urlcolor=blue]{hyperref}
\usepackage[dvipsnames]{xcolor}
\usepackage{mathtools}
\usepackage{graphicx}
\usepackage{extarrows}

\usepackage[normalem]{ulem}
\usepackage{calc}

\usepackage{accents}
\newcommand{\dbtilde}[1]{\accentset{\approx}{#1}}
\newtheorem{theorem}{Theorem}[section]
\newtheorem{lemma}[theorem]{Lemma}

\newtheorem{proposition}[theorem]{Proposition}

\theoremstyle{definition}

\newtheorem{remark}[theorem]{Remark}
\numberwithin{equation}{section}

\newcommand{\eqw}{\xlongequal{w^*}}

\newcommand{\N}{\mathbb{N}}
\newcommand{\K}{\mathbb{K}}

\DeclareMathOperator{\Id}{Id}

\newcommand{\nn}[1]{{\left\vert\kern-0.25ex\left\vert\kern-0.25ex\left\vert #1 
		\right\vert\kern-0.25ex\right\vert\kern-0.25ex\right\vert}}

\renewcommand{\geq}{\geqslant}
\renewcommand{\leq}{\leqslant}

\newcommand{\NA}{\operatorname{NA}}
\newcommand{\NRA}{\operatorname{NRA}}

\newcommand{\pten}{\ensuremath{\widehat{\otimes}_\pi}}

\newcommand{\eps}{\varepsilon}

\begin{document}
\setcounter{tocdepth}{1}


\title{A note on numerical radius attaining mappings} 

\author[Jung]{Mingu Jung}
\address[Jung]{School of Mathematics, Korea Institute for Advanced Study, 02455 Seoul, Republic of Korea \newline
\href{http://orcid.org/0000-0003-2240-2855}{ORCID: \texttt{0000-0003-2240-2855} }}
\email{\texttt{jmingoo@kias.re.kr}}
\urladdr{\url{https://clemg.blog/}}

\thanks{Mingu Jung was supported by a KIAS Individual Grant (MG086601) at Korea Institute for Advanced Study.}



\begin{abstract}  
We prove that if every bounded linear operator (or $N$-homogeneous polynomials) on a Banach space $X$ with the compact approximation property attains its numerical radius, then $X$ is a finite dimensional space. Moreover, we present an improvement of the polynomial James’ theorem for numerical radius proved by Acosta, Becerra Guerrero and Gal\'an \cite{AGG2003}. Finally, the denseness of weakly (uniformly) continuous $2$-homogeneous polynomials on a Banach space whose Aron-Berner extensions attain their numerical radii is obtained. 
\end{abstract}
\maketitle


\section{Introduction} 
Let $X$ and $Y$ be Banach spaces and denote by $\mathcal{L}(X,Y)$ (resp., $\NA(X,Y)$) the space of all bounded linear operators (resp., the set of all \emph{norm attaining operators}) from $X$ into $Y$. Recall that a bounded linear operator $T \in \mathcal{L}(X,Y)$ is said to be a \emph{norm attaining operator} if the operator norm $\|T\|$ is equal to $\max \{ \|T(x) \| : \|x\| \leq 1\}$. When $X=Y$, we simply denote them by $\mathcal{L}(X)$ and $\NA(X)$. One of the famous unsolved problems in Banach space theory is to characterize when the space $\mathcal{L}(X,Y)$ coincides with the set $\NA(X,Y)$, and some partial solutions for this problem have been given by J. Holub \cite{H2} and J. Mujica \cite{Muj}. To the best of our knowledge, the most general result, which is obtained recently in \cite{DJM}, states in particular the following: if $X$ is a reflexive space and $Y$ is an arbitrary Banach space for which either $X$ or $Y$ has the \emph{compact approximation property}, then $\mathcal{L}(X, Y) = \NA(X,Y)$ is equivalent to that every operator from $X$ into $Y$ is a compact operator. Thus, for a Banach space $X$ with the compact approximation property, the situation when $\mathcal{L}(X) = \NA(X)$ implies that $X$ must be finite dimensional. 

Paralleling the study on norm attaining operators, there have been numerous studies on \emph{numerical radius attaining operators} \cite{Acosta1991, AR, AP1989, AP1989-2, BS, Cardassi, Paya}. Recall that the \emph{numerical radius of an operator} $T \in \mathcal{L}(X)$ is defined by 
\begin{equation}\label{eq:nu}
\nu(T) := \sup \big\{ |x^*(T(x))|: (x, x^*) \in \Pi(X) \big\},
\end{equation}
where $\Pi(X) := \{ (x, x^*) \in S_X \times S_{X^*}: x^*(x) = 1 \}$. We say that $T \in \mathcal{L}(X)$ is a \emph{numerical radius attaining operator} if there exists some element $(x,x^*) \in \Pi(X)$ such that $\nu(T) = |x^* (T(x))|$. We denote by $\NRA(X)$ the set of all numerical radius attaining operators on $X$. 
Among others, it is proved in \cite{AR} that if all the rank-one operators on a Banach space $X$ attain their numerical radii, then $X$ must be reflexive, which can be viewed as a version of the celebrated James' theorem \cite{James} for numerical radius. 

One of the main motivations of this paper is to investigate Banach spaces $X$ for which $\mathcal{L}(X) = \NRA(X)$. In Section \ref{sec:1}, we shall observe that for a Banach space $X$, if $\mathcal{L}(X) = \NRA(X)$, then $X$ must be a separable Banach space, which is the numerical radius version of Kalton's result \cite{K}. Moreover, we prove that if $X$ has the compact approximation property and $\mathcal{L}(X) = \NRA(X)$, then $X$ must be a finite dimensional space. 
Following that, we are devoted to the case of $N$-homogeneous polynomials. As it is defined in \eqref{eq:nu}, the \emph{numerical radius of an $N$-homogeneous polynomial} $P$ from a Banach space $X$ into itself is given by 
\[
\nu (P) := \sup \big\{ |x^*(P(x))|: (x, x^*) \in \Pi(X) \big\}. 
\]
We say that $P$ is a \emph{numerical radius attaining $N$-homogeneous polynomial} if $\nu(P) = |x^* (P(x))|$ for some $(x,x^*) \in \Pi(X)$. Let us denote by $\mathcal{P}(^N X)$ and $\NRA(^N X)$ the Banach space of $N$-homogeneous polynomials from $X$ into $X$ and the set of all numerical radius attaining $N$-homogeneous polynomials, respectively. In Section \ref{sec:2}, we give an improvement of the polynomial James' theorem for numerical radius proved in \cite{AGG2003} by using the refinement of James' theorem in \cite{JiM}. Similarly as in the previous section, we prove that for a Banach space $X$ with the compact approximation property, if $\mathcal{P}(^N X) = \NRA(^N X)$, then $X$ must be a finite dimensional space. 
Finally, Section \ref{sec:3} focuses on the denseness of the set of weakly (uniformly) continuous $2$-homogeneous polynomials whose Aron-Berner extensions attain their numerical radii. This can be viewed as a numerical radius version of the result in \cite{CLS2010} where it is proved that the set of $2$-homogeneous polynomials whose Aron-Berner extensions attain their norms is dense.

\section{Numerical radius attaining operators}\label{sec:1}
Let $X$ be a Banach space. Then it is clear that $\nu (T) \leq \|T\|$ for every $T \in \mathcal{L}(X)$ and that $\nu$ is a seminorm on $\mathcal{L}(X)$. The greatest constant $k \geq 0$ such that $k \|T\| \leq \nu(T)$ for every $T \in \mathcal{L}(X)$ is the well known constant which is called the \emph{numerical index} of $X$ and denoted by $n(X)$. 
Equivalently, 
\begin{equation*}
n(X) = \inf \big\{ \nu(T): T \in \mathcal{L}(X), \|T\| = 1 \big\}. 
\end{equation*}
If we consider 
\begin{equation*}
\mathcal{Z}(X) := \big\{ S \in \mathcal{L}(X): \nu (S) = 0 \big\},
\end{equation*}
which is a closed subspace of $\mathcal{L}(X)$, and the quotient space $\mathcal{L}(X)/\mathcal{Z}(X)$ endowed with the \emph{norm} $\nu(T+\mathcal{Z}(X)) = \inf \{ \nu (T-S): S \in \mathcal{Z}(X)\}$, then it is straightforward to check the following:
\begin{itemize}
\itemsep0.3em
	\item[(a)] $\nu (T+\mathcal{Z}(X)) = \nu(T)$ for every $T \in \mathcal{L}(X)$, and 
	\item[(b)] $\nu (T+\mathcal{Z}(X)) \leq \|T + \mathcal{Z}(X)\| = \inf \{ \|T - S\|: S \in \mathcal{Z}(X)\}$ for every $T \in \mathcal{L}(X)$.
\end{itemize}

For Banach spaces $X$ and $Y$, recall that $\mathcal{L}(X,Y^*)$ is the dual of the projective tensor product space $X \pten Y$. In our scenario, we shall observe that for a reflexive Banach space $X$ the normed space $(\mathcal{L}(X)/\mathcal{Z}(X), \nu)$ is the dual of a kind of tensor product space.
To this end, consider the following normed space 
\begin{equation*}
X \otimes_{\Pi} X^* := \Big\{ \sum_{i=1}^n \lambda_i x_i \otimes x_i^*: n \in \N, \ \lambda_i \in \K, \ (x_i, x_i^*) \in \Pi(X) \Big\} 
\end{equation*} 
endowed with the norm 
\begin{equation}\label{eq:Pi-norm}
\|u\| := \inf \Big\{ \sum_{i=1}^n |\lambda_i|: u = \sum_{i=1}^n \lambda_i x_i \otimes x_i^*,  \ (x_i, x_i^*) \in \Pi(X), \ \lambda_i \in \K, \ n \in \N \Big\}. 
\end{equation}
We denote by $X \widehat{\otimes}_{\Pi} X^*$ the completion of $X \otimes_{\Pi} X^*$ with respect to the norm in \eqref{eq:Pi-norm}. Let, as usual, $\mathcal{B} (X \times X^*)$ stand for the set of all bilinear forms on $X \times X^*$ and consider the seminorm $\| \cdot \|_{\Pi}$ on $\mathcal{B} (X \times X^*)$ given by $\| B \|_{\Pi} := \sup \{ |B(x,x^*)|: (x,x^*) \in \Pi (X) \}$.

\begin{theorem}\label{thm:isometry}
Let $X$ be a reflexive Banach space. Then 
\[
(\mathcal{L}(X)/\mathcal{Z}(X), \nu) \stackrel{1}{=} (\mathcal{B}(X \times X^*), \|\cdot\|_{\Pi}) / \ker \| \cdot \|_{\Pi} \stackrel{1}{=} (X \widehat{\otimes}_{\Pi} X^*)^* \,\,\, (\text{isometric isomorphically}).
\]
\end{theorem} 

\begin{proof}
We claim that the mapping $\Phi : (\mathcal{L}(X)/\mathcal{Z}(X), \nu) \rightarrow (\mathcal{B}(X \times X^*), \|\cdot\|_{\Pi}) / \ker \| \cdot \|_{\Pi}$ given by 
\[
\Phi (T + \mathcal{Z} (X)) := \varphi_T + \ker \| \cdot \|_{\Pi},
\]
where $\varphi_T (x,x^*) = x^* (T(x))$, is an isometric isomorphism. It is not difficult to check that $\Phi$ is a well-defined linear operator. Note that for any $S \in \mathcal{Z}(X)$, we have $\varphi_S \in \ker \| \cdot \|_{\Pi}$. Conversely, if $\psi \in \ker \| \cdot \|_{\Pi}$, then the operator $L_{\psi} \in \mathcal{L}(X)$ given by $L_{\psi} (x) = \psi (x, \cdot)$ belongs to $\mathcal{Z}(X)$.  Moreover, 
\begin{align*}
\| \Phi (T + \mathcal{Z} (X)) \| = \| \varphi_{T} + \ker\| \cdot\|_{\Pi} \| &= \inf \{ \| \varphi_T + \psi \|_\Pi : \psi \in \ker \| \cdot \|_{\Pi} \} \\
&= \inf \{ \| \varphi_T + \varphi_S\|_{\Pi} : S \in \mathcal{Z} (X) \} \\
&= \inf \{ \nu (T+S) : S \in \mathcal{Z} (X)\} = \nu (T + \mathcal{Z} (X)). 
\end{align*}
To see that $\Phi$ is surjective, take any $B \in \mathcal{B}(X\times X^*)$. As $X$ is reflexive, the operator $T_B$ given by $T_B (x) = B(x,\cdot)$ belongs to $\mathcal{L}(X)$. It is straightforward to check that $\Phi (T_B + \mathcal{Z}(X))= B + \ker \| \cdot\|_{\Pi}$. 
Thus, the claim is proved. 

Next, define $\Psi : (\mathcal{B}(X \times X^*), \|\cdot\|_{\Pi}) / \ker \| \cdot \|_{\Pi} \rightarrow (X \widehat{\otimes}_{\Pi} X^*)^*$ by 
\[
\Psi (\varphi + \ker \| \cdot \|_{\Pi} ) (u) := \sum_{i=1}^n \lambda_i \varphi (x_i, x_i^*) 
\]
where $u = \sum_{i=1}^n \lambda_i x_i \otimes x_i^*$. The map $\Psi$ is a well-defined linear operator. Observe that 
\begin{align*}
| \Psi (\varphi + \ker \| \cdot\|_{\Pi}) (u) | &= \Big | \sum_{i=1}^n \lambda_i \varphi (x_i, x_i^*) \Big| =\Big | \sum_{i=1}^n \lambda_i (\varphi + \xi) (x_i, x_i^*) \Big| \leq \sum_{i=1}^n |\lambda_i| \| \varphi + \xi \|
\end{align*} 
for every $\xi \in \ker \| \cdot \|_{\Pi}$ and any representation $u = \sum_{i=1}^n \lambda_i x_i \otimes x_i^*$. It follows that 
\[
| \Psi (\varphi + \ker \| \cdot\|_{\Pi}) (u) | \leq \| \varphi + \xi\| \|u\| \text{ for every } \xi \in \ker \| \cdot \|_{\Pi};
\]
hence $\|\Psi (\varphi + \ker \| \cdot\|_{\Pi})\| \leq \| \varphi + \ker \| \cdot \|_{\Pi}\|$. On the other hand,
\begin{align*}
\|\Psi (\varphi + \ker \| \cdot\|_{\Pi})\| &\geq \sup \{ | \Psi (\varphi + \ker \| \cdot\|_{\Pi}) (x \otimes x^*)| : (x, x^*) \in \Pi (X) \} \\
&= \sup \{ | \varphi (x, x^*)| : (x, x^*) \in \Pi (X) \} \\
&= \| \varphi  \|_{\Pi}.
\end{align*} 
Thus, $\|\Psi (\varphi + \ker \| \cdot\|_{\Pi})\| \geq \| \varphi + \ker \| \cdot \|_{\Pi}\|$. 
It remains to check that $\Psi$ is surjective. Let $L \in (X \widehat{\otimes}_{\Pi} X^*)^*$ be fixed. Note that $X {\otimes}_{\Pi} X^*$ is an algebraic subspace of the projective tensor product space $X \widehat{\otimes}_{\pi} X^*$. Considering an algebraic complement of $X {\otimes}_{\Pi} X^*$, we can consider the algebraic linear extension $\tilde{L}$ of $L \vert_{X {\otimes}_{\Pi} X^*}$ to $X \widehat{\otimes}_{\pi} X^*$ (which is not necessarily continuous). Now, consider the map $\dbtilde{L} : X \times X^* \rightarrow \mathbb{K}$ given by $\dbtilde{L}(x,x^*) = \tilde{L} (x \otimes x^*)$ which is a bilinear form on $X \times X^*$. Note that 
\[
|\dbtilde{L} (x,x^*)| = |\tilde{L}(x \otimes x^*)| = |L(x\otimes x^*)| \leq \|L\| \text{ for every } (x,x^*) \in \Pi(X); 
\]
hence $\dbtilde{L}$ is a member of $(\mathcal{B}(X \times X^*), \|\cdot\|_{\Pi})$ with $\| \dbtilde{L} \|_{\Pi} \leq \| L \|$. As it is clear that $\Psi (\dbtilde{L} + \ker \| \cdot \|_{\Pi}) = L$, the surjectivity of $\Psi$ is proved.
\end{proof}

\begin{remark}
We remark that the above Theorem \ref{thm:isometry} generalizes the observation in the proof of \cite[Proposition 23]{DKLM-21}, where the authors showed that if $X$ is a reflexive Banach space with $n(X) >0$, then $(\mathcal{L}(X), \nu)$ is isometrically isomorphic to $(\mathcal{B}(X \times X^*), \|\cdot\|_{\Pi})$ and to $(X \widehat{\otimes}_{\Pi} X^*)^*$. 
\end{remark}

\begin{proposition}\label{prop:reflexive} Let $X$ be a Banach space. If $\mathcal{L}(X) = \NRA(X) $, then $(\mathcal{L}(X)/\mathcal{Z}(X), \nu)$ is reflexive. If, in addition, $n(X) >0$, then $\mathcal{L}(X)$ is reflexive. 
\end{proposition}

\begin{proof} 
Notice that $X$ is reflexive due to \cite[Theorem 1]{AR} since, in particular, every rank-one operator attains its numerical radius.
By Theorem \ref{thm:isometry}, we have that $(\mathcal{L}(X)/\mathcal{Z}(X), \nu)$ and $(X \widehat{\otimes}_{\Pi} X^*)^*$ are isometrically isomorphic. 
Take $L \in (X \widehat{\otimes}_{\Pi} X^*)^*$ and consider $T+\mathcal{Z}(X)$ the corresponding element in $(\mathcal{L}(X)/\mathcal{Z}(X), \nu)$. 
Say $\nu (T) = |x^* (T(x))|$ for some $(x,x^*) \in \Pi (X)$. Then 
\[
|L(x\otimes x^*)| = |x^* (T(x))| = \nu (T) = \nu (T + \mathcal{Z}(X)) = \| L \|,
\]
which implies that $L$ attains its norm. Since $L$ is chosen arbitrarily, this shows that $X \widehat{\otimes}_{\Pi} X^*$ is reflexive by James' theorem \cite{James}; hence so is $(\mathcal{L}(X)/\mathcal{Z}(X), \nu)$.
\end{proof}

N.J. Kalton proved in \cite[Theorem 2]{K} that if $\mathcal{L} (X)$ is reflexive, then $X$ must be a separable (reflexive) space. Having this fact in mind, one direct consequence of Proposition \ref{prop:reflexive} is that if $\mathcal{L}(X) = \NRA(X)$ and $n(X) >0$, then $X$ must be separable. However, the following result shows that the assumption that $n(X) >0$ is in fact superfluous. The proof is motivated by the one of aforementioned result of N.J. Kalton. 

\begin{proposition} Let $X$ be a Banach space. If $\mathcal{L}(X) = \NRA(X) $, then $X$ must be a separable reflexive space. 
\end{proposition}

\begin{proof} As before, the reflexivity of $X$ follows from \cite[Theorem 1]{AR}. 
Thus, it is enough to show that $X$ is separable. Assume that $X$ is non-separable. By a result due to Chadwick \cite{C}, there exists non-trivial projections $Q_n$ on $X$ such that $\sum_{n=1}^{\infty} Q_n = \Id_X$ in the weak operator topology. Note that $\sup_{N \in \mathbb{N}} \left\| \sum_{n=1}^N Q_n \right\| < \infty$ by the Principle of Uniform Boundedness. Since $\mathcal{L}(X) = \NRA(X) $, we have from Theorem \ref{thm:isometry} and Proposition \ref{prop:reflexive} that $(\mathcal{L}(X) /\mathcal{Z}(X), \nu)$ is reflexive and its dual is isometrically isomorphic to $X \widehat{\otimes}_{\Pi} X^*$. Observe that $\langle u, \sum_{n=1}^k Q_n +\mathcal{Z}(X)  \rangle \rightarrow \langle u, \sum_{n=1}^\infty Q_n + \mathcal{Z}(X) \rangle$ for each $u \in X \widehat{\otimes}_{\Pi} X^*$ as $k \rightarrow \infty$, that is, $\sum_{n=1}^{\infty} Q_n + \mathcal{Z}(X) = \Id_X + \mathcal{Z}(X)$ in the weak topology. It follows that there exists $S_n = \sum_{i=1}^n \lambda_i^{(n)} Q_i$ such that $S_n + \mathcal{Z}(X) \stackrel{\nu}{\longrightarrow} \Id_X + \mathcal{Z}(X)$ in $(\mathcal{L}(X) /\mathcal{Z}(X), \nu)$. Let us find a large $n \in \N$ so that 
	\begin{equation*}
	\nu (S_n - \Id_X) = \nu ( (S_n + \mathcal{Z}(X)) - (\Id_X + \mathcal{Z}(X)) )  < 1. 
	\end{equation*} 
Note that $Q_{n+1} \neq 0$ and take $x_0 \in X$ such that $Q_{n+1} (x_0) \in S_X$. Put $x_1 = Q_{n+1} (x_0)$ and let $x_1^* \in S_X^*$ so that $x_1^* (x_1) = 1$. Then we have that 
\[
|x_1^* ((S_n - \Id_X) (Q_{n+1} (x_0))) | = |x_1^* ((S_n - \Id_X) (x_1))| \leq \nu (S_n - \Id_X) < 1.
\]
On the other hand, 
\[
|x_1^* ((S_n - \Id_X) (Q_{n+1} (x_0))) | = |x_1^* ( S_n ( Q_{n+1} (x_0) ) )  - x_1^* (x_1) | = |x_1^* (x_1)| = 1
\]
since $Q_{n+1} (X) \subseteq \ker S_n$. This is a contradiction, so $X$ must be separable.
\end{proof}

As already noted, if $\mathcal{L}(X) =\NRA(X)$, then $X$ must be reflexive. However, the converse is false. Indeed, it is observed also in \cite{AR} that given a reflexive Banach space $X$ with a Schauder basis, there exists an isomorphic copy $\tilde{X}$ of $X$ such that not every operator on $\tilde{X}$ is numerical radius attaining. The following result, in particular, shows that if $X$ is an infinite dimensional Banach space with Schauder basis, then there exists $T \in \mathcal{L}(X)$ which does not attain its numerical radius without considering renormings of $X$. 

\begin{theorem} \label{theorem:CAP+NRA-fin-dim} Let $X$ be a Banach space with the compact approximation property. If $\mathcal{L}(X) = \NRA(X) $, then $X$ must be finite dimensional. 
\end{theorem}

\begin{proof} Assume that $X$ is infinite dimensional. As $X$ is reflexive, applying the Josefson-Nissenzweig theorem (see \cite[\S~XII]{Diestel_seq_series}), we take a sequence $(x_n)_{n=1}^{\infty} \subseteq S_X$ converging weakly to $0$. For each $n \in \N$, take $x_n^* \in S_{X^*}$ to be such that $x_n^*(x_n) = 1$. 
Note from Theorem \ref{thm:isometry} and Proposition \ref{prop:reflexive} that $X \widehat{\otimes}_{\Pi} X^*$ is reflexive and its dual is $(\mathcal{L}(X) /\mathcal{Z}(X), \nu)$. Passing to a subsequence if it is necessary, we can assume that $(x_n \otimes x_n^*)$ converges weakly to some $u \in X \widehat{\otimes}_{\Pi} X^*$. Observe that for any compact operator $T$ on $X$, 
	\begin{equation}\label{eq:cpt}
	\langle u, T + \mathcal{Z}(X) \rangle = \lim_n \big\langle x_n \otimes x_n^*, T + \mathcal{Z}(X) \big\rangle = \lim_n x_n^*(T(x_n)) = 0.
	\end{equation}

Now, let $0 < \eps < 1/2$ and pick $u_0 = \sum_{j=1}^N \lambda_j v_j \otimes v_j^* \in X \widehat{\otimes}_{\Pi} X^*$ with $(\lambda_j) \in \mathbb{K}$ and $(v_j,v_j^*) \in \Pi(X)$ for every $j=1,\ldots, N$, to be such that $\|u - u_0\|_{\Pi} < \eps$. Notice from \eqref{eq:cpt} that 
\begin{equation} \label{eq3}
| \langle u_0, T + \mathcal{Z}(X) \rangle| \leq | \langle u, T + \mathcal{Z}(X) \rangle| + \eps = \eps, \ \forall \ T \in \mathcal{K}(X) \text{ with } \|T\| \leq 1.
\end{equation}
Since $X$ is a reflexive space with the compact approximation property, $X$ has in fact the metric compact approximation property (see, for instance, \cite[Proposition 1 and Remarks 1]{CJ}). Thus, there exists a net $(T_{\alpha})$ of norm-one compact operators such that $(T_{\alpha}) \rightarrow \Id_X$ in the compact open topology. It follows, therefore, that
\begin{eqnarray*}
| \langle u_0, \Id_X + \mathcal{Z}(X) \rangle &=& \Big| \sum_{j=1}^N \lambda_j v_j^*(\Id_X(v_j)) \Big| \\
&=& \lim_{\alpha} \Big| \sum_{j=1}^N \lambda_j v_j^*(T_{\alpha} (v_j)) \Big| =  \lim_{\alpha} \left| \langle u_0, T_{\alpha} + \mathcal{Z}(X) \rangle \right| \stackrel{(\ref{eq3})}{\leq} \eps.	
\end{eqnarray*} 
This implies that 
\begin{equation*}
| \langle u, \Id_X + \mathcal{Z}(X) \rangle | \leq | \langle u_0, \Id_X + \mathcal{Z}(X) \rangle| + \eps \leq 2 \eps < 1. 
\end{equation*}
On the other hand,
\begin{equation*}
\langle u, \Id_X + \mathcal{Z}(X) \rangle = \lim_n \langle x_n \otimes x_n^*, \Id_X + \mathcal{Z}(X) \rangle = \lim_n x_n^*(x_n) = 1,
\end{equation*}
which is a contradiction. Therefore, $X$ must be finite dimensional. 	
\end{proof} 

It is natural to ask if every \emph{compact operator} on a Banach space $X$ (with the compact approximation property) attains its numerical radius, then $X$ must be finite dimensional. However, it is known \cite{Acosta1991, dGS} that every compact operator on $\ell_p$ with $1<p<\infty$ attains its numerical radius. Therefore, there is no compact operator version of Theorem \ref{theorem:CAP+NRA-fin-dim}. 

\section{Numerical radius attaining homogeneous polynomials}\label{sec:2}

Let $X$ be a Banach space and $N \in \mathbb{N} \cup \{0\}$ be given.
It is known that if every (finite-type) $N$-homogeneous polynomial on $X$ attains its numerical radius, then $X$ must be a reflexive space. 
More precisely, what is proved in \cite{AGG2003} is the following. For $x^*, x_1^*, \ldots, x_N^* \in \mathbb{N}$ and $x_0 \in X$, let the notation $P_{x_1^*, \ldots, x_N^*, x^*, x_0}$ stand for the $(N+1)$-homogeneous polynomial on $X$ defined by 
\[
P_{x_1^*, \ldots, x_N^*, x^*; x_0} (x) = x_1^* (x) \cdots x_N^* (x) x^* (x) x_0.  
\]
\begin{theorem}[\mbox{\cite[Theorem 4]{AGG2003}}]\label{thm:AGG2003}
Let $X$ be a Banach space and $N \in \mathbb{N}\cup \{0\}$ be given. If there exist $x_0 \in X \setminus \{0\}$ and $x_1^*, \ldots, x_{N}^* \in X^* \setminus \{0\}$ such that $P_{x_1^*, \ldots, x_N^*, x^*, x_0}$ attains its numerical radius for every $x^* \in X^*$, then $X$ must be a reflexive space. 
\end{theorem} 
Nevertheless, there is a reflexive Banach space on which not every $N$-homogeneous polynomial is numerical radius attaining \cite[Example 1]{AGG2003}. 
We start this section by generalizing the above Theorem \ref{thm:AGG2003} by using a refined version of James' theorem \cite{JiM} as it is done for the case of linear operators in \cite{AGG2004}. 

\begin{theorem}
Let $X$ be a Banach space and $N \in \mathbb{N}$ be given. Suppose that there exist $x_0 \in X \setminus \{0\}$ and $x_1,\ldots, x_{N-1}^* \in X^* \setminus \{0\}$ such that the set
\[
\{ x^* \in X^* : P_{x_1^*,\ldots,x_{N-1}^*, x^*; x_0} \text{ attains its numerical radius.} \} 
\]
has a non-empty weak-star interior. Then $X$ must be reflexive. 
\end{theorem}

\begin{proof}
Write the set 
\[
B = \{ x_1^* (z) \cdots x_{N-1}^* (z) z^* (x_0) z : (z, z^*) \in \Pi (X) \} \subseteq X. 
\]
Then $P_{x_1^*,\ldots,x_{N-1}^*, x^*, x_0}$ attains its numerical radius if and only if $|x^*|$ attains its supremum on the set $B$. This implies that the following set 
\[
\{ x^* \in X^* : |x^*| \text{ attains its supremum on } B \} 
\]
has a non-empty weak-star interior. By the same reasoning as in the proof of \cite[Theorem 3.1]{AGG2004}, it is enough to show that $\overline{\mathbb{D}} B:= \{ \lambda b : \lambda \in \mathbb{K}, |\lambda| \leq 1, b \in B \}$ contains an open ball which, in turn, implies that there is an equivalent norm on $X$ which makes the set of norm attaining functionals contain a weak-star open set (see \cite[Proposition 3.2]{JiM}). 

Put $Q := x_1^* \cdots x_{N-1}^*$. If $Q(x_0) \neq 0$, then take any $\delta >0$ such that $|Q(x_0)|>\delta$. If not, then observe that by applying the identity principle \cite[Proposition 5.7]{Mujicabook} for any $\eps >0$ there exists $u \in X$ such that $\| u - x_0\| < \eps$ and $|Q(u)| > \delta$ for some $\delta >0$. 
Since the first case can be deduced from the second case, we only consider the second case. 

Let $0< \eps < \|x_0\|$, and find $u \in X$ and $\delta>0$ such that $\| u- x_0\| < \eps$ and $|Q(u)|>\delta$. Take $0<s<1$ so that 
\begin{equation}\label{eq:s}
\frac{s (\|x_0\| + \eps)^N }{(\|x_0\| -\eps) \delta} < \frac{1}{3}. 
\end{equation} 

Pick an element $y \in s u + r B_X$, where $0<r<1$ is sufficiently small so that 
\begin{equation}\label{eq:r1}
2s^{-1} r < \| x_0 \| - \eps,
\end{equation} 
\begin{align}\label{eq:r2} 
|Q(y)| &= |Q(su + rz)| \quad (\text{for some } z \in B_X ) \\ \nonumber
&\geq s^{N-1} |Q(u)| - |f(r)| > s^{N-1} \delta - |f(r)| > 0, 
\end{align} 
where $f(r) = Q(su + rz) - s^{N-1} Q(u)$ (so, $|f(t)| \rightarrow 0$ as $t \rightarrow 0$), and 
\begin{equation}\label{eq:r3}
\left| \frac{(s (\|x_0\| + \eps) + r )^N}{(s^{N-1} \delta - |f(r)|)(\|x_0\| - 2 s^{-1}r -\eps)} - \frac{s (\|x_0\| + \eps)^N }{(\|x_0\| -\eps) \delta} \right| < \frac{1}{3}. 
\end{equation} 
Take $y^* \in S_{Y^*}$ so that $y^* (y) = \|y\|$, so $(\frac{y}{\|y\|}, y^*) \in \Pi (X)$. 
Note that 
\begin{align*}
|y^* (x_0)| = |y^* (u+ x_0 - u)| &\geq |y^*(u)| - \|x_0 - u \| \\
&> |y^* (s^{-1} y - s^{-1} r z)| - \eps \\
&\geq s^{-1} \|y\| - s^{-1} r - \eps \geq \| x_0\| - 2 s^{-1} r - \eps > 0.
\end{align*} 
Since 
\[
y^* (x_0) y = \frac{\|y\|^N}{\prod_{i=1}^{N-1} x_i^* (y)} \Big[ \prod_{i=1}^{N-1} x_i^* \Big( \frac{y}{\|y\|}\Big) \Big] y^* ( x_0) \frac{y}{\| y\|} \in \left( \frac{\|y\|^N}{Q(y)} \right)  B 
\]
and $y^* (x_0) \neq 0$, we conclude that 
\[
y \in \left( \frac{\|y\|^N}{Q(y) y^* (x_0)} \right)  B. 
\]
Now, observe from \eqref{eq:s}-\eqref{eq:r3} and the fact $\|y\| \leq s(\|x_0\| +\eps) +r$ that 
\begin{align*}
\left| \frac{\|y\|^N}{Q(y) y^* (x_0)} \right| 
&\leq \frac{(s (\|x_0\| + \eps) + r )^N}{(s^{N-1} \delta - f(r))(\|x_0\| - 2 s^{-1}r -\eps)} < \frac{s (\|x_0\| + \eps)^N }{(\|x_0\| -\eps) \delta} + \frac{1}{3} < \frac{2}{3}.
\end{align*} 
This proves that $su + r B_X$ is contained in $\overline{\mathbb{D}} B$ and completes the proof. 
\end{proof}

Next, we will obtain the polynomial version of Theorem \ref{theorem:CAP+NRA-fin-dim}. In order to do so, we follow a similar approach as in the preceding section. Given a Banach space $X$ and $N \in \mathbb{N}$, consider the closed subspace $\mathcal{Z} (^N X):= \{ P \in \mathcal{P}(^N X) : \nu(P) = 0\}$. Then the quotient space $(\mathcal{P}(^N X) / \mathcal{Z} (^N X), \nu)$ turns to be a normed space endowed with $\nu (P + \mathcal{Z}(^N X)) = \inf \{ \nu (P-Q): Q \in \mathcal{Z}(^N X)\}$. 
Also, consider the following space of tensors: 
\[
(\otimes^N X ) \otimes_{\Pi} X^* := \Big\{ \sum_{i=1}^n \lambda_i (\otimes^N x_i ) \otimes x_i^* : (x_i,x_i^*) \in \Pi (X), \lambda_i \in \mathbb{K}, n \in \mathbb{N} \Big\},  
\]
endowed with the norm 
\[
\|u \| = \inf \Big \{ \sum_{i=1}^n |\lambda_i| : u = \sum_{i=1}^n \lambda_i (\otimes^N x_i ) \otimes x_i^* : (x_i,x_i^*) \in \Pi (X), \lambda_i \in \mathbb{K}, n \in \mathbb{N} \Big\}
\]
and denote by $(\otimes^N X ) \widehat{\otimes}_{\Pi} X^*$ its completion. The following duality result is a version of Theorem \ref{thm:isometry} for homogeneous polynomials.   

\begin{proposition}\label{prop:polyisometry}
Let $X$ be a reflexive Banach space and $N \in \mathbb{N}$. Then $(\mathcal{P}(^N X) / \mathcal{Z} (^N X), \nu)$ is isometrically isomorphic to $((\otimes^N X ) \widehat{\otimes}_{\Pi} X^*)^*$. 
\end{proposition} 

\begin{proof}
Even though the proof is very similar to that of Theorem \ref{thm:isometry}, we sketch the proof for the sake of completeness.  
Let $\Phi : (\mathcal{P}(^N X) / \mathcal{Z} (^N X), \nu) \rightarrow ((\otimes^N X ) \widehat{\otimes}_{\Pi} X^*)^*$ be the map defined as 
\[
\Phi (P + \mathcal{Z}(^N X)) (u) = \sum_{i=1}^n \lambda_i x_i^* (P(x_i))
\]
for $u = \sum_{i=1}^n \lambda_i (\otimes^N x_i) \otimes x_i^*$. Then $\Phi$ is well-defined and it is a linear operator. To see that $\Phi$ is a surjective isometry, observe that 
\begin{align*}
\| \Phi (P + \mathcal{Z}(^N X)) (u)\| = \Big | \sum_{i=1}^n  \lambda_i x_i^* (P(x_i)) \Big| \leq \nu(P) \sum_{i=1}^n |\lambda_i| 
\end{align*} 
for any representation $u = \sum_{i=1}^n \lambda_i (\otimes^N x_i)\otimes x_i^*$. Since $\nu(P) = \nu(P+\mathcal{Z}(^N X))$, we conclude that 
$\| \Phi (P+\mathcal{Z}(^N X))\| \leq \nu (P + \mathcal{Z}(^N X))$. Conversely, note that 
\[
|x^* (P(x))| = | \Phi (P + \mathcal{Z}(^N X) ) ( (\otimes x^N) \otimes x^* ) | \leq \| \Phi (P + \mathcal{Z}(^N X) ) \| 
\]
for every $(x,x^*) \in \Pi (X)$; hence $\nu (P + \mathcal{Z}(^N X )) = \nu (P) \leq  \| \Phi (P + \mathcal{Z}(^N X) ) \|$. This shows that $\Phi$ is an isometry. Finally, we claim that $\Phi$ is a surjective. Let $L \in ((\otimes^N X ) \widehat{\otimes}_{\Pi} X^*)^*$ be given. Let $\tilde{L}$ be an algebraic linear extension of $L \vert_{(\otimes^N X ) {\otimes}_{\Pi} X^*}$ to $(\widehat{\otimes}_{N,s,\pi} X) \widehat{\otimes}_{\pi} X^*$ and define the map $\dbtilde{L}$ from $X$ into $X^{**} = X$ by 
\[
(\dbtilde{L}(x) )(x^*) = \tilde{L} ((\otimes^N x) \otimes x^*) \text{ for every } x \in X \text{ and } x^* \in X^*. 
\]
Then $\dbtilde{L}$ is an $N$-homogeneous polynomial on $X$ satisfying that $\nu (\dbtilde{L}) \leq \| L \|$ and $\Phi (\dbtilde{L} + \mathcal{Z}(^N X)) = L$. This completes the proof. 
\end{proof} 

Arguing in the same way as in Proposition \ref{prop:reflexive}, we obtain the following.  

\begin{proposition}\label{prop:reflexive-poly}
Let $X$ be a Banach space and $N \in \mathbb{N}$. If $\mathcal{P}(^N X) = \NRA(^N X)$, then $(\mathcal{P}(^N X)/\mathcal{Z}(^N X), \nu)$ is reflexive. 
\end{proposition} 

In \cite{CGKM2006}, the \emph{polynomial numerical index of order N}, denoted by $n^{(N)} (X)$, of a Banach space $X$ is introduced and investigated. Namely,
\[
n^{(N)} (X) = \{ \nu (P) : P \in \mathcal{P} (^N X),  \|P\| = 1  \}. 
\]
As a consequence, Proposition \ref{prop:reflexive-poly} implies that if $n^{(N)} (X) > 0$ and $\mathcal{P}(^N X) = \NRA(^N X)$, then $\mathcal{P}(^N X)$ is reflexive since $\mathcal{Z} (^N X) = \{ 0\}$. It might be worth mentioning that $n^{(N)}(X) > 0$ for every $N \in \mathbb{N}$ whenever $n(X) >0$ (see \cite[Proposition 2.5]{CGKM2006}). 

Now, we are ready to prove the following. The idea of the proof is similar to that of Theorem \ref{theorem:CAP+NRA-fin-dim}, but using slightly more carefully the Josefson-Nissenzweig theorem. Recall that a polynomial $P \in \mathcal{P}(^N X)$ is said to be \emph{weakly sequentially continuous} if the sequence $(P(x_n))$ is norm-convergent whenever a sequence $(x_n)$ is weakly convergent. Also, a polynomial $P \in \mathcal{P}(^N X)$ is said to be \emph{weakly (uniformly) continuous} if it is weakly uniformly continuous on bounded subsets of $X$. 

\begin{theorem}\label{theorem:CAP+NRA-fin-dim-poly}
Let $X$ be a Banach space with the compact approximation property and $N \in \mathbb{N}$. If $\mathcal{P}(^N X) = \NRA(^N X)$, then $X$ must be finite dimensional. 
\end{theorem} 

\begin{proof}
Assume that $X$ is an infinite dimensional Banach space. Note from Theorem \ref{thm:AGG2003} that $X$ is reflexive. Using the Josefson-Nissenzweig theorem, take a sequence $(x_n) \subseteq S_X$ converging weakly to some $x_\infty \in X$ with $\|x_\infty\| = 1/2$. For each $n \in \N$, pick $x_n^* \in S_{X^*}$ to be such that $x_n^*(x_n) = 1$, and choose $x_\infty^* \in S_{X^*}$ such that $x_\infty^* (x_\infty) = 1/2$. 
By Proposition \ref{prop:reflexive-poly}, $(\otimes^N X) \widehat{\otimes}_{\Pi} X^*$ is reflexive and its dual is $(\mathcal{P}(^N X)/\mathcal{Z}(^N X), \nu)$. Thus, we may assume that $(\otimes^N x_n) \otimes x_n^*$ converges weakly to some $u \in (\otimes^N X) \widehat{\otimes}_{\Pi} X^*$. Then for every weakly sequentially continuous polynomial $P$ in $\mathcal{P}(^N X)$ with $\|P\|=1$, we have that 
\begin{equation}\label{eq:u,P}
|\langle u, P + \mathcal{Z}(^N X ) \rangle| = \lim_n | x_n^* (P(x_n)) | = \lim_n |x_n^* ( P(x_\infty)) | \leq \frac{1}{2^N}.  
\end{equation} 
Now, let $Q \in \mathcal{P}(^N X)$ with $\|Q\|=1$ be given. Using the metric compact approximation property of $X$, we can take a net $(P_\alpha)$ of weakly (uniformly) continuous (hence, weakly sequentially continuous) $N$-homogeneous polynomials on $X$ converging to $Q$ in the compact-open topology and satisfying $\|P_\alpha\| \leq 1$ (see \cite[Corollary 7]{Caliskan2004} and \cite[Proposition 2.1]{MV}). Since $u$ can be approximated by elements of the form $\sum_{i=1}^n \lambda_i (\otimes^N x_i) \otimes x_i^*$ with $(x_i, x_i^*) \in \Pi(X)$, we conclude that 
\begin{equation*}
|\langle u, Q + \mathcal{Z}(^N X) \rangle| \leq \limsup_\alpha |\langle u, P_\alpha + \mathcal{Z}(^N X) \rangle| \stackrel{\eqref{eq:u,P}}{\leq} \frac{1}{2^N}. 
\end{equation*}
As the above inequality holds for arbitrary $Q \in \mathcal{P}(^N X)$ with $\|Q \| = 1$, we have that 
\begin{equation}\label{eq:uQ}
\sup \{ | \langle u, Q + \mathcal{Z}(^N X) \rangle | : Q \in \mathcal{P}(^N X), \|Q\| =1 \} \leq \frac{1}{2^N}. 
\end{equation} 
On the other hand, consider $P_\infty \in \mathcal{P}(^N X)$ given by $P_\infty (x) = x_\infty^* (x)^{N-1} x$ for every $x \in X$. It is clear that  $\|P_\infty\| =1$. However, 
\[
\langle u, P_\infty + \mathcal{Z}(^N X) \rangle = \lim_n x_n^* ( P_\infty (x_n)) = \lim_{n} x_\infty^* (x_n)^{N-1} x_n^* (x_n) = x_\infty^* (x_\infty)^{N-1} = \frac{1}{2^{N-1}}.
\]
This contradicts \eqref{eq:uQ}, so we conclude that $X$ must be a finite dimensional space. 
\end{proof} 

In the previous section, we mentioned that every compact operator on $\ell_p$, $1<p<\infty$, attains its  numerical radius. Not surprisingly, this result can be extended to the case of weakly sequentially continuous homogeneous polynomials. It is immediate that a weakly sequentially continuous bounded linear operator on $X$ is nothing but a completely continuous operator on $X$. The proof is similar to the argument in \cite{Acosta1991}, but we present the details for the sake of completeness. 

\begin{proposition}\label{prop:lp-poly}
Given $1<p<\infty$ and $N \in \mathbb{N}$, every weakly sequentially continuous $N$-homogeneous polynomial on $\ell_p$ is numerical radius attaining. 
\end{proposition} 

\begin{proof}
Let $P$ be a weakly sequentially continuous $N$-homogeneous polynomial on $\ell_p$ and take $(x_n, x_n^*) \in \Pi (\ell_p)$ so that $|x_n^* (P(x_n))| \rightarrow \nu(P) \neq 0$. Note that $x_n^* (i) = |x_n (i)|^{p/q} \alpha_n (i)$, where $\alpha_n (i) \in \mathbb{T}$ satisfies that $|x_n (i)| = \alpha_n (i) x_n (i)$ and $q$ is the conjugate of $p$. Passing to a subsequence, we may assume that $(x_n)$ converges weakly to $x_\infty \in B_{\ell_p}$. Consider $x_\infty^* \in \ell_p^* = \ell_q$ given by $x_\infty^* (i) = |x_\infty (i)|^{p/q} \alpha (i)$, where $\alpha(i) \in \mathbb{T}$ satisfies that $|x_\infty (i)| = \alpha(i) x_\infty (i)$. Then one can deduce that $(x_n^*)$ converges weekly to $x_\infty^*$ and $x_\infty^* (x_\infty) = \| x_\infty^*\| \|x_\infty\|$. Moreover, since $P$ is weakly sequentially continuous, $|x_\infty^* (P(x_\infty)) | = \nu(P) \neq 0$. In particular, $x_\infty \neq 0$ and $x_\infty^* \neq 0$. Thus,
\[
\frac{\nu(P)}{\|x_\infty^* \| \|x_\infty\|^N } =  \frac{ |x_\infty^*(P(x_\infty) ) |}{\|x_\infty^*\| \|x_\infty\|^N }  = \left| \frac{x_\infty^*}{\|x_\infty^*\|} \Big ( P \Big ( \frac{x_\infty}{\|x_\infty\|} \Big) \Big )\right|  \leq \nu(P);
\] 
hence $\|x_\infty\| = \|x_\infty^* \| = 1$ and $P$ attains its numerical radius at $(x_\infty, x_\infty^*) \in \Pi (\ell_p)$. 
\end{proof}

One may ask whether a similar result can be obtained for Lipschitz functions on a Banach space $X$ into itself as there is a notion of numerical radius for Lipschitz functions, the so called \emph{Lipschitz numerical radius} (we refer the interested readers to \cite{CJT, KMMW, WHT}). However, it is not possible to consider the case when every Lipschitz function attains its Lipschitz numerical radius. Indeed, it is proved in \cite{CJT} that for any Banach space $X$ the set of Lipschitz numerical radius attaining Lipschitz functions on $X$ is not dense in the whole space of Lipschitz functions on $X$.

\section{$2$-homogeneous polynomials whose Aron-Berner extensions attain their numerical radii}\label{sec:3}

It is shown in \cite{AP1989} that for any Banach space $X$ the set of bounded linear operators on $X$ whose second adjoints attain their numerical radii is dense in $\mathcal{L}(X)$. Moreover, it is mentioned in \cite{AP1989-2} that, actually, it is true that the set of bounded linear operators on $X$ whose first adjoints attain their numerical radii is dense. As a matter of fact, those results are parallel versions of the result of J. Lindenstrauss \cite{Lin} and V. Zizler \cite{Zizler} on norm attaining operators. 
In the context of homogeneous polynomials, it is first observed in \cite{AGM} that the set of {scalar-valued} $2$-homogeneous polynomials whose Aron-Berner extensions attain their norms is dense in the whole space. Afterwards, this result was extended to a vector-valued $2$-homogeneous polynomials in \cite{CLS2010}. 

Recall that the canonical extension of a bilinear mapping is obtained by weak-star density as follows (which also works for a general multilinear mapping) \cite{Arens}. For Banach spaces $X, Y$ and $Z$, and $B \in \mathcal{B}(X\times Y; Z)$, the space of all bounded bilinear mappings from $X \times Y$ into $Z$, the extension $\overline{B} \in \mathcal{B}(X^{**} \times Y^{**}; Z^{**})$ of $B$ is defined by 
\[
\overline{B} (x^{**}, y^{**}) \eqw \lim_\alpha \lim_\beta B(x_\alpha, y_\beta)
\]
where $(x_\alpha) \subseteq X$ and $(y_\beta) \subseteq Y$ are nets converging weak-star to $x^{**} \in X^{**}$ and $y^{**} \in Y^{**}$, respectively. The \emph{Aron-Berner extension} of $P \in \mathcal{P}(^2 X)$ is the polynomial $AB(P) \in \mathcal{P}(^2 X^{**})$ given by $AB(P) (x^{**}) := \overline{B}(x^{**})$ where $B$ is the unique symmetric bilinear mapping from $X \times X$ into $X$ associated to $P$ \cite{AronBerner}. It is known that $\|P \| = \|AB(P)\|$ for every $P \in \mathcal{P}(^2 X)$ \cite{DG}.

Let us denote by $\mathcal{P}_{wu}(^N X)$ the space of weakly (uniformly) continuous $N$-homogeneous polynomials from $X$ into $X$, that is, those elements are {weakly (uniformly) continuous on bounded subsets of $X$}. Notice that the space $\mathcal{P}_{wu} (^1 X)$ coincides with the space of all compact operators on $X$.

In this section, we prove that for any Banach space $X$ the set of $P \in \mathcal{P}_{wu}(^2 X)$ whose Aron-Berner extensions are numerical radius attaining is dense in $\mathcal{P}_{wu}(^2 X)$. In order to prove this, we need the following lemma which is a slight modification of \cite[Lemma 3]{AP1989}. 

\begin{lemma}\label{lem:AcoPaya}
Let $X$ be a Banach space and $P \in \mathcal{P}(^2 X)$. Suppose that $|u^* (P(u))| > \nu(P) - \alpha$ for some $(u,u^*) \in \Pi (X)$ and $\alpha >0$. Let us define 
\[
P' (x) = P(x) + \lambda \delta^2 u^* (x)^2 u + \lambda \delta^2 (u^* (B(u,x)) )^2 u, 
\]
where $\lambda \in \mathbb{T}$ satisfies that $u^* (P(u)) = \lambda |u^* (P(u))|$ and $B$ is the symmetric bilinear mapping corresponding to $P$. If $|y^* ( P' (y))| > \nu (P') - \rho$ for some $(y,y^*) \in \Pi(X)$ and $\rho >0$, then we have 
\[
\delta^2 |u^* (y)|^2 + \delta^2 |u^* (B(u, y))|^2 + \rho \geq -\alpha + \delta^2 + \delta^2 (\nu(P)-\alpha)^2. 
\]
\end{lemma} 

\begin{proof}
Note first that 
\begin{align*}
\nu(P') \geq | u^* (P' (u))| &= | u^* (P(x)) + \lambda \delta^2 + \lambda \delta^2 u^* (P(u))^2 | \\
&= |u^* (P(u))| + \delta^2 + \delta^2 |u^* (P(u))|^2 \stackrel{\text{(I)}}{\geq} \nu(P) -\alpha + \delta^2 + \delta^2 (\nu(P) - \alpha)^2. 
\end{align*} 
Second, observe that 
\begin{align*}
\nu(P') &< |y^* (P'(y))| + \rho \\
&\stackrel{\text{(II)}}{\leq} \nu (P) + \delta^2 |u^* (y)|^2 + \delta^2 |u^* (B(u,y))|^2 + \rho. 
\end{align*} 
Combining (I) with (II), we complete the proof. 
\end{proof}

The proof of the following result is essentially based on an adaption of Lindenstrauss' argument in \cite[Theorem 1]{Lin}. 

\begin{theorem}\label{thm:bilinear}
Let $X$ be a Banach space. Then the set 
\[
\{ P \in \mathcal{P}_{wu}(^2 X) : AB(P) \in \NRA (^2 X^{**}) \} 
\]
is dense in $\mathcal{P}_{wu}(^2 X)$. 
\end{theorem} 

\begin{proof}
Let $P \in \mathcal{P}_{wu}(^2 X)$, $\| P \|=1$ and $B$ be the corresponding symmetric bilinear mapping. Given $0< \eps<1$, choose decreasing sequences $(\alpha_n)$ and $(\delta_n)$ of positive numbers satisfying that 
\begin{equation}\label{eq:delta1}
\sum_{j=1}^n (1+4^2) \delta_j^2 < \eps, \quad  \frac{1}{\delta_n^2} \sum_{j=n+1}^\infty (1+4^2) \delta_j^2 \rightarrow 0 \quad \text{and} \quad \frac{\alpha_n}{\delta_n^2} \rightarrow 0. 
\end{equation} 
Put $P_1 := P$ and inductively construct sequences $(P_n)$ in $\mathcal{P}(^2 X)$ and $(x_n, x_n^*) \in \Pi (X)$ satisfying 
\begin{align*}
&|x_n^* (P_n (x_n))| > \nu (P_n) - \alpha_n, \\
&P_{n+1} (x) := P_n (x) + \lambda_n \delta_n^2 x_n^* (x)^2 x_n + \lambda_n \delta_n^2 (x_n^* (B_n (x, x_n)))^2 x_n, 
\end{align*} 
where $\lambda_n \in \mathbb{T}$ satisfies that $x_n^* (P_n (x_n)) = \lambda_n |x_n^* (P_n (x_n))|$ and $B_n$ is the symmetric bilinear mappings corresponding to $P_n$. Note that 
\[
\|P_2 \| \leq \| P_1 \| + \delta_1^2 + \delta_1^2 \|S_1\|^2 \leq 1 + (1+2^2 ) \delta_1^2. 
\]
Also, 
\begin{align*}
\|P_3\| \leq \|P_2\| + \delta_2^2 + \delta_2^2 \|S_2\|^2 &\leq 1 + (1+2^2 ) \delta_1^2 + \delta_2^2 + \delta_2^2 (2^2 (1 + (1+2^2 ) \delta_1^2 )^2 ) \\
&\leq 1 + (1+2^2 ) \delta_1^2 + (1 + 4^2) \delta_2^2, 
\end{align*} 
since $(1 + (1+2^2 ) \delta_1^2 ) < 2$. In this way, one can verify that 
\[
\|P_{n+1}\| \leq 1 + (1+2^2 ) \delta_1^2 + (1 + 4^2) \delta_2^2 + \cdots + (1 +4^2) \delta_{n}^2 \leq 1 + \sum_{j=1}^n (1+4^2) \delta_j^2 \leq 2 
\]
for each $n \in \mathbb{N}$. This also shows that 
\[
\| P_{n+k} - P_n \| \leq \sum_{j=n}^{n+k-1} (1+\|B_j\|^2) \delta_j^2  \leq \sum_{j=n}^{n+k-1} (1+4^2) \delta_j^2
\]
since $\|B_j\| \leq 2 \|P_j\| \leq 2^2$. Thus, $(P_n)$ converges in norm to some $P_\infty \in \mathcal{P}(^2 X)$, and $\|P_\infty - P \| \leq \eps$. Notice that each $P_{n}$ is weakly uniformly continuous on bounded sets; hence so is $P_\infty$. 

We claim that $AB(P_\infty) \in \NRA (^2 X^{**})$. Note that 
\begin{align*}
|x_{n+k}^* (P_{n+1} (x_{n+k}))| &\geq \nu (P_{n+k}) - \| P_{n+k}-P_{n+1}\| \\
&\geq \nu (P_{n+1}) - 2 \|P_{n+k}-P_{n+1}\| \geq \nu (P_{n+1}) - 2 \sum_{j=n}^{n+k-1} (1+4^2) \delta_j^2. 
\end{align*} 
Applying Lemma \ref{lem:AcoPaya} with $P'=P_{n+1}$, we obtain that 
\[
\delta_n^2 |u_n^* (x_{n+k})|^2 + \delta_n^2 |x_n^* (B_n(x_n, x_{n+k}))|^2 + 2 \sum_{j=n}^{n+k-1} (1+4^2) \delta_j^2 \geq -\alpha_n + \delta_n^2 + \delta_n^2 (\nu(P_n)-\alpha_n)^2. 
\] 
Let $z$ and $\phi$ be cluster points of the sequences $(x_n)$ and $(x_n^*)$ in the weak-star topologies of $X^{**}$ and $X^{***}$, respectively. Letting $k \rightarrow \infty$, we have 
\begin{equation}\label{eq:delta_n}
\delta_n^2 |z(u_n^*)|^2 + \delta_n^2 |[ \overline{B_n} (x_n, z)] (x_n^*) |^2 + 2 \sum_{j=n}^{\infty} (1+4^2) \delta_j^2 \geq -\alpha_n + \delta_n^2 + \delta_n^2 (\nu(P_n)-\alpha_n)^2, 
\end{equation} 
where $\overline{B_n}$ is the canonical extension of $B_n$. 
Note from the polarization inequality \cite[Theorem 2.2]{Mujicabook} that $\overline{B_n}$ converges in norm to $\overline{B_\infty}$, where $B_\infty$ is the symmetric bilinear mapping corresponding to $P_\infty$ and $\overline{B_{\infty}}$ is its canonical extension. 

Notice that $AP(P_\infty)$ is uniformly continuous on $(B_{X^{**}}, w^*)$. Applying the classical polarization formula \cite[Theorem 1.10]{Mujicabook}, we observe that $\overline{B_\infty} (x_n, z)$ converges in norm to $\overline{B_\infty} (z,z) = AB(P_\infty) (z)$. By dividing $\delta_n^2$ in \eqref{eq:delta_n}, letting $n \rightarrow \infty$ and combining it with \eqref{eq:delta1}, we obtain that 
\begin{equation}\label{eq:phi_z_AB_P}
|\phi(z)|^2 + |\phi (AB(P_\infty) (z))|^2 \geq 1 + \nu(P_\infty)^2. 
\end{equation}
This shows that $|\phi(z)| = 1$ and $|\phi (AB(P_\infty) (z))| = \nu(P_\infty)$. Since $\nu(AB(P_\infty)) = \nu(P_\infty)$ \cite[Corollary 2.14]{CGKM2006}, we conclude that $AB(P_\infty)$ attains its numerical radius. 
\end{proof}

\begin{remark}
If we denote by $\mathcal{P}_{wsc}(^N X)$ the space of {weakly sequentially continuous} $N$-homogeneous polynomials (for its definition, see the paragraph before Theorem \ref{theorem:CAP+NRA-fin-dim-poly}), then we have the following:
\[
\mathcal{P}_{wu} (^N X) \subseteq \mathcal{P}_{wsc}(^N X) \subseteq \mathcal{P} (^N X).
\]
We do not know if we can replace $\mathcal{P}_{wu} (^N X)$ in Theorem \ref{thm:bilinear} by $\mathcal{P}_{wsc}(^N X)$, since for a given $P \in \mathcal{P}_{wsc}(^N X)$, its extension $AB(P) \in \mathcal{P}(^N X^{**})$ is not necessarily weak-star to norm sequentially continuous (hence the inequality in \eqref{eq:phi_z_AB_P} is not clear). As a matter of fact, it can be deduced from \cite[Example 1.9]{Zal} that there exist a Banach space $X$ and a weakly sequentially continuous $2$-homogeneous polynomial $P$ on $X$ such that $AB(P)$ is not weak-star to norm sequentially continuous. 
\end{remark}

It is observed in \cite{CLM2012} that if the dual space of a Banach space $X$ is separable and has the approximation property, then for any natural number $N \in \mathbb{N}$, the set of $N$-homogeneous polynomials from $X$ to a dual Banach space $Y^*$ whose Aron-Berner extensions attain their norms is dense. Indeed, their idea was to use the integral representation for elements in tensor product spaces by identifying  polynomials between $X$ and $Y^*$ with elements in $C(B_{X^{**}} \times B_{Y^{**}})$ and using the Riesz representation theorem \cite[Theorem 2.2]{CLM2012} from the compactness of $B_{X^{**}} \times B_{Y^{**}}$. Here, $B_{X^{**}}$ and $B_{Y^{**}}$ are endowed with their weak-star topologies. In order to adapt this to our situation, it would be natural to consider first the set $\Pi (X^{*})$ instead of $B_{X^{**}} \times B_{Y^{**}}$. However, the set $\Pi (X^*)$ cannot be a compact subset of $B_{X^*} \times B_{X^{**}}$, both are endowed with their weak-star topologies, unless $X$ is finite dimensional as the following result shows.  

\begin{proposition}
Let $X$ be a Banach space.
\begin{enumerate}
\setlength\itemsep{0.3em}
\item 
The set $\Pi (X^*)$ is a compact subset of $B_{X^*} \times B_{X^{**}}$ if and only if $X$ is finite  dimensional, where both $B_{X^*}$ and $B_{X^{**}}$ are endowed with their weak-star topologies.   
\item The set $\Pi (X)$ is a compact subset of $B_X \times B_{X^*}$ if and only if $X$ is finite  dimensional, where $B_X$ and $B_{X^*}$ are endowed with the weak topology and weak-star topology, respectively.
\end{enumerate} 
\end{proposition} 

\begin{proof}
(1): As ``If" part is clear, assume to the contrary that $X$ is infinite dimensional. By applying the Josefson-Nissenzweig theorem, take $(x_n^*) \subseteq S_{X^*}$ so that $(x_n^*)$ converges weak-star to some $x_\infty^* \in X^*$ with $\|x_\infty^*\| = 1/2$. Take $x_n^{**} \in S_{X^{**}}$ so that $x_n^{**} (x_n^*) = 1$ for each $n \in \mathbb{N}$. If $\Pi (X^*)$ were compact, then we would have a subnet $(x_\alpha^*, x_\alpha^{**})$ which converge to some $(z^*, z^{**})$ in $\Pi (X^*)$. This implies, in particular, that $z^* = x_\infty^*$ which is a contradiction.

(2): We prove the ``only if" part. Assuming $X$ is infinite dimensional, take a sequence $(x_n^*) \subseteq S_{X^*}$ so that $(x_n^*)$ converges weak-star to some $x_\infty^* \in X^*$ with $\|x_\infty^*\| = 1/2$ as above. Take $x_n \in S_X$ such that $|1- x_n^* (x_n)| < (1/\sqrt{2}n)^2$ for each $n \in \mathbb{N}$. By the Bishop-Phelps-Bollob\'as theorem \cite{Bol}, there exists $(y_n, y_n^*) \in \Pi (X)$ such that $\|y_n -x_n\| < 1/n + (1/n)^2$ and $\| y_n^* - x_n^* \| < 1/n$. If $\Pi (X)$ were compact, then there would be a subnet $(y_\alpha, y_\alpha^{*})$ which converge to some $(z, z^{*})$ in $\Pi (X)$. This implies that $z^* = x_\infty^*$ which contradicts $\|z^*\| = 1$.
\end{proof} 

\subsection*{Acknowledgment}
The author is grateful to Sheldon Dantas, Miguel Mart\'in and \'Oscar Rold\'an.

\end{document}